\documentclass[11pt]{article}

\RequirePackage[OT1]{fontenc}
\RequirePackage{amsthm,amsmath,amssymb,amsfonts}
\RequirePackage[numbers]{natbib}
\RequirePackage[colorlinks,citecolor=blue,urlcolor=blue]{hyperref}
\usepackage{amstext} 
\usepackage{array}   
\usepackage{enumitem} 
\usepackage{xcolor}
\setlist[itemize]{noitemsep} 
\allowdisplaybreaks
\usepackage{multirow}
\usepackage[flushleft]{threeparttable}
\usepackage{verbatim}
\usepackage{hyperref}
\usepackage{bm,bbm}
\usepackage{longtable}
\usepackage{rotating}
\usepackage{caption}
\usepackage{indentfirst}
\usepackage{subfigure}
\usepackage{booktabs}

\usepackage{anysize}

\marginsize{1.1in}{1.1in}{0.55in}{0.8in}

\numberwithin{equation}{section}
\theoremstyle{plain}
\newtheorem{theorem}{Theorem}[section]
\newtheorem{lemma}{Lemma}[section]

\newtheorem{proposition}{Proposition}[section]
\newtheorem{corollary}[theorem]{Corollary}
\newtheorem{remark}{Remark}[section]

\newtheorem{definition}{Definition}
\theoremstyle{remark}\newtheorem{assumption}{Assumption}
\theoremstyle{remark}
\DeclareMathOperator*{\argmin}{arg\,min}

\newcommand{\RR}{\mathbb{R}}

\newcommand{\ZZ}{\mathbb{Z}}
\newcommand{\LL}{\mathbb{L}}

\newcommand{\N}{\mathrm{N}}
\renewcommand{\P}{\mathrm{P}}
\newcommand{\E}{\mathrm{E}}
\renewcommand{\d}{\mathrm{d}}
\newcommand{\bj}{\bm{j}}
\newcommand{\bJ}{\bm{J}}

\newcommand{\bX}{\bm{X}}
\newcommand{\bx}{\bm{x}}

\newcommand{\cF}{\mathcal{F}}

\newcommand{\Ind}{\mathbbm{1}}
\newcommand{\ceil}[1]{\left\lceil #1 \right\rceil}
\newcommand{\abs}[1]{\left| #1 \right|}
\newcommand{\pnorm}[2]{\lVert#1\rVert_{#2}}

\begin{document} 
	
\title{Posterior Contraction and Testing for Multivariate Isotonic Regression
\thanks{
The research was supported in part by NSF Grant number DMS-1916419.
}}

\author{Kang Wang \thanks{Department of Statistics, North Carolina State Universtiy. Email: kwang22@ncsu.edu}
\and
Subhashis Ghosal\thanks{Department of Statistics, North Carolina State Universtiy. Email: sghosal@stat.ncsu.edu}}

\date{}
\maketitle
\begin{abstract}
\noindent
 We consider the nonparametric regression problem with multiple predictors and an additive  error, where the regression function is assumed to be coordinatewise nondecreasing. We propose a Bayesian approach to make an inference on the multivariate monotone regression function, obtain the posterior contraction rate, and construct a universally consistent Bayesian testing procedure for multivariate monotonicity. To facilitate posterior analysis, we set aside the shape restrictions temporarily, and endow a prior on blockwise constant regression functions  with heights independently normally distributed. The unknown variance of the error term is either estimated by the marginal maximum likelihood estimate, or is equipped with an inverse-gamma prior. Then the unrestricted block-heights are a posteriori also independently normally distributed given the error variance, by conjugacy. To comply with the shape restrictions, we project samples from the unrestricted posterior onto the class of multivariate monotone functions, inducing the ``projection-posterior distribution'', to be used for making an inference. Under an $\LL_1$-metric, we show that the projection-posterior based on $n$ independent samples contracts around the true monotone regression function at the optimal rate $n^{-1/(2+d)}$. Then we construct a Bayesian test for multivariate monotonicity based on the posterior probability of a shrinking neighborhood of the class of multivariate monotone functions. We show that the test is universally consistent, that is, the level of the Bayesian test goes to zero, and the power at any fixed alternative goes to one. Moreover, we show that for a  smooth alternative function, power goes to one as long as its distance to the class of multivariate monotone functions is at least of the order of the estimation error for a smooth function. To the best of our knowledge, there is no other test for multivariate monotonicity available in the Bayesian or the frequentist literature.
 \vskip 0.3cm

\noindent
{\bf Keywords:} Isotonic regression;
Multivariate isotonic regression;
Bayesian tests for multivariate monotonicity;
Contraction rate.
\end{abstract}

\section{Introduction}
\label{sec:introduction}
Shape restricted inference is an important nonparametric statistical technique with a long history. Functions with qualitative shape restrictions, like monotonic functional relationship between variables, are quite common in natural sciences, sociology, economics and many other areas. Shape restrictions on the function space can also serve as the relaxation to restricted parametric models, such as log-concave density estimation. The shape constraints themselves yield function estimators with good statistical properties without resorting to subjective selection of the smoothness level, such as in kernel or spline smoothing. Starting from early works on statistical inference under order restrictions, problems with monotonicity constraints on parameters of interest, regression functions and probability densities were extensively studied. For the univariate monotone function estimation problem, the least squares estimator for an isotonic regression function and the maximum likelihood estimator for a decreasing density function have interesting geometrical representations respectively as the slope of the  greatest convex minorant and the least concave majorant of a cumulative sum diagram. The limit distribution at an interior point on which the function has positive derivative is well known as the rescaled Chernoff's distribution; see Grenander \cite{Grenander1956}, Prakasa Rao \cite{PrakasaRao1969}, Brunk \cite{Brunk1970}, Groeneboom \cite{Groeneboom1985, groeneboom1989brownian}, Barlow et al. \cite{barlow1972statistical}, and Robertson et al. \cite{robertson1988order}. Asymptotic global behaviors of the least squares estimators under monotone constraints are well developed with respect to various metrics; see Groeneboom \cite{Groeneboom1985}, Kukilov and Lopuha\"a  \cite{Kulikov2005}, Durot \cite{durot2007}, and Durot et al. \cite{Durot2012}. Zhang \cite{zhang2002risk} and Bellec \cite{bellec2018sharp} studied the non-asymptotic risk bounds of the least squares estimators. Testing of the monotonicity was studied in the univariate case by Akakpo et al. \cite{akakpo2014testing}, Hall and Heckman \cite{hall2000testing} and Ghosal et al. \cite{ghosal2000testing}.
Applications of shape restricted inference in various areas, like causal inference, genetics, material science are still of growing interest; more details can be found in Westling et al. \cite{Westling2020}, Luss and Rosset \cite{luss2014generalized} and Vittorietti et al. \cite{Vittorietti2021}.

Compared to the well-studied case of univariate monotone shape restricted inference, convergence results for multivariate monotonicity were lacking until recent years. Among different possible multivariate monotonicity restrictions, the coordinatewise monotonicity is popularly considered. This naturally arises in some modeling contexts studied in Robertson et al. \cite{robertson1988order},  Saarlera and Arjas \cite{saarela2011method} and Fokianos et al.  \cite{Fokianos2020}. In the frequentist literature, the least squares estimator under the multivariate coordinatewise monotonicity  constraint has received the most attention. For both a fixed grid design or a random design, the minimax rate is given by $n^{-\min\{2/(d+2), 1/d\}}$ with respect to the squared empirical $\LL_2$-metric when the true regression function is coordinatewise nondecreasing and is of bounded variation (Chatterjee et al. \cite{chatterjee2018matrix}, Han  et al. \cite{han2019isotonic}). 
Han \cite{han2021set} showed that some special global empirical risk minimizers, such as the least squares estimator in multivariate isotonic regression, are rate optimal even when the entropy integral concerned therein diverges rapidly. 
For $d\ge 2$ and $p\geq 1$, the minimax risk under the general empirical $\LL_p$-loss $n^{-1} \sum_{i=1}^n \E |\hat f(x_i)-f(x_i)|^p$ of the estimator $\hat f$ of a function $f$ for deterministic predictors on a grid, and the integrated $\LL_p$-risk $\int \E |\hat f(x)-f(x)|^p dG(x)$ for a random predictor $X\sim G$, are bounded below by a multiple of $n^{-\min\{1/d, p/(d+2)\}}$ under some conditions on the signal-to-noise ratio and the error term; see Deng and Zhang \cite{deng2020isotonic}. Bagchi and Dhar \cite{bagchi2020study} derived the asymptotic distribution of the least square estimator of a multivariate monotone regression function. 
In addition to the least squares estimator or the empirical risk minimization estimators, other estimators, such as a block-estimator modifying the min-max formula for the isotonic least squares solution (Robertson et al.  \cite{robertson1988order}), have been proposed and studied; see Fokianos et al. \cite{Fokianos2020}, Deng and Zhang \cite{deng2020isotonic} and Han and Zhang \cite{han2020}.
The computation of isotonic regression with respect to a general partial ordering minimizing the  $\LL_q$-metric also attracted attention. One solution is to put this question under the framework of convex optimization with linear constraints; see, for example, Dykstra and Robertson \cite{Dykstra1982}, de Leeuw \cite{Leeuwisotone} and Meyer \cite{Meyer2013}. A sequential partitioning algorithm is designed for isotonic regression under the weighted $\LL_1$-metric that computes in $O(n\log n)$ time for the coordinatewise isotonic regression with $2$-dimensional grid designs and in $O(n^2\log n)$ time for the $d\ge 3$ case; see Stout \cite{Stout2013} for details. Stout \cite{Stout2015} gave another algorithm with better computation time under the $\LL_1$-metric for the unweighted data. In terms of the $\LL_2$-metric, which leads to the usual isotonic least squares estimator, the algorithm in Spouge et al. \cite{Spouge2003} can compute in $O(n^2)$ time for a two-dimensional grid data.

Bayesian approaches to isotonic regression are also available in the literature. Most of these approaches deal with a univariate isotonic regression function. Neelon and Dunson \cite{Neelon2004} modeled the regression function as a piecewise linear function and incorporated the monotonicity constraints in the priors of the sequential slopes. Shivley et al. \cite{shively2009bayesian} 
considered Bayesian regression splines under the monotonicity constraint, which is incorporated into the spline coefficients through a mixture of a constrained normal distribution and a probability distribution on the boundary of the constrained parameter set. Lin and Dunson \cite{lin2014bayesian} considered a Gaussian process prior, and projected posterior samples on the space of monotone functions to obtain an induced posterior distribution, which is subsequently used to make inference. Chakraborty and Ghosal  \cite{chakraborty2021convergence,chakraborty2020coverage,ChakrabortyDensity} used the same idea with a piecewise constant prior and obtained results on posterior contraction and frequentist coverage of Bayesian credible intervals. For the multivariate monotone regression, Saarela and Arjas \cite{saarela2011method} used marked point processes to construct piecewise constant sample paths for the function. They considered a homogeneous Poisson process prior on the random point positions and endowed the associated marks, the function value at the point, with the uniform prior supported on the allowed interval restricted by the shape constraints. Chipman et al. \cite{chipman2021mbart} applied a constrained sum-of-trees to model monotone regression functions. To obtain posterior samples, Markov chain Monte Carlo (MCMC) methods are used for each method mentioned above.
Bayesian testing procedure of the monotonicity in the univariate case has also been proposed by a few authors. Salomond \cite{salomond2014adaptive} and Chakraborty and Ghosal \cite{chakraborty2021convergence, ChakrabortyDensity} developed tests based on the posterior distribution of a discrepancy of the function from the unrestricted posterior with its monotone projection. Scott et al. \cite{scott2015} used smoothing splines and regression splines to model the regression function, and incorporated the monotonicity constraints into the prior for the coefficients. For testing the monotonicity, they considered the Bayes factor and converted the monotonicity hypothesis to a condition on the minimum of the derivative functions. To the best of our knowledge, no test for multivariate monotonicity, Bayesian or frequentist, is yet available in the literature. 

In this paper, we consider a Bayesian approach to multivariate monotone regression, using the projection technique. We show that the  resulting induced posterior supported on block-wise constant multivariate monotone function contracts at the optimal rate with respect to an $\LL_1$-metric. The basis of the result is a new $\LL_1$-approximation result for multivariate monotone functions by piecewise constant functions. We then construct a test for multivariate monotonicity based on the posterior probability of a slight enlargement of the set of multivariate monotone functions. We show that the resulting Bayesian test is universally consistent in that the size of the test goes to zero, and the power goes to one at any fixed alternative, as the sample size increases to infinity.  
We further show that, even for alternatives approaching the null region, the power can go to one, provided that the alternative maintains a distance at least a sufficiently large multiple of the posterior contraction rate determined by its smoothness. These results generalize the testing results of Chakraborty and Ghosal \cite{chakraborty2021convergence} to the multidimensional predictors. 

The rest of this paper is organized as follows. In Section \ref{sec:setup}, we describe the prior distribution and the projection-posterior approach. Posterior contraction rates and the properties of the Bayesian test for monotonicity are presented in Section \ref{sec:main}. Simulation studies to judge the qualities of the proposed estimation and testing procedure in finite sample sizes are conducted in Section \ref{sec:simulation}. Proofs are deferred to Section \ref{sec:proofs}. Certain auxiliary results and their proofs are presented in the appendix.

\section{Setup, prior and posterior}
\label{sec:setup}

We shall use the following notations and symbols throughout the paper. The notation $\RR$ stands for the real line, $\ZZ$ for the set of integers, $\ZZ_>$ for the set of positive integers. Vectors and matrices will be denoted by bold letters, and the default form of a vector is assumed to be in the column form. For $\bm{a}\in\RR^d$, let $a_k$ denote the $k$th coordinate, $k=1,\ldots, d$. The symbols $\bm{1}$ and $\bm{0}$ will respectively denote the $d$-dimensional all-one and all-zero vectors. For a real $x$, $\lfloor x \rfloor$ (respectively, $\lceil x\rceil$) will stand for the greatest integer less (respectively, smallest integer greater) than or equal to $x$. The indicator function of a set $A$ is denoted by $\Ind_A(\cdot)$. For $p>0$, let $\LL_p(\mu)$ denote the set of real-valued functions defined on $[0,1]^d$ with respect to a measure $\mu$ whose $p$th power is integrable. For $p\geq 1$ and $f\in \LL_p(\mu)$, the $\LL_p$-norm of $f$ is denoted by $\pnorm{f}{p,\mu}$. For a distance $\rho$ on functions, a function $f$ and a set of functions $\cF$, let $\rho(f,\cF)=\inf\{\rho(f,g): g\in \cF\}$. The symbol $\lesssim$ will stand for an inequality up to a constant multiple, and $\asymp$ will stand for equality in order. For two positive real sequences, $a_n$ and $b_n$, we also say $a_n \gg b_n$ if $b_n = o(a_n)$. 
Let $\N(\nu,\sigma^2)$ stand for the normal distribution with mean $\nu$ and variance $\sigma^2$.

Consider the natural partial ordering $\preceq$  on $\RR^d$ given by: $\bx_1\preceq \bx_2$ if $x_{1,k}\leq x_{2,k}$ for every $1\leq k \leq d$ and $(\bx_1,\bx_2)\in \RR^d\times \RR^d$ and we also use $\bx_2\succeq \bx_1$ if $\bx_1\preceq \bx_2$. 
For $(\bj_1,\bj_2)\in \ZZ^d\times\ZZ^d$ and $\bj_1\preceq\bj_2$, let  $[\bj_1:\bj_2] = \{\bj\in \ZZ^d: \bj_1\preceq \bj \preceq \bj_2\}$. 

\begin{definition}
	\label{def:mmonotone}
	A function $f: I\to \RR$, where $I\subset \RR^d$, is called multivariate monotone if $f(\bx_1)\le f(\bx_2)$ whenever $\bx_1\preceq \bx_2$. 
\end{definition}

The space of all multivariate monotone functions on $[0,1]^d$ will be denoted by $\mathcal{M}$.  

We consider the nonparametric multivariate regression model 
\begin{align}
	\label{formula:model}
	Y=f(\bX)  + \varepsilon,
\end{align}
where $\bX$ is the $d$-dimensional predictor and $\varepsilon$ is an error term with zero mean and finite variance, independent of $\bX$. We shall assume, essentially without loss of generality, that the domain of $\bX$ is $[0,1]^d$. Instead of a traditional smoothness assumption on the regression function $f$, we assume that $f$ is multivariate monotone. 

We observe the data $\mathbb{D}_n$ consisting of $n$ samples $(\bX_1,Y_1),\ldots,(\bX_n,Y_n)$ independently from the model. The predictor variable $\bX$ may be deterministic, or may be obtained independently from a fixed distribution $G$, independent of the random error variable $\varepsilon$. 
To make inference on $f$, we adopt a Bayesian approach by putting an appropriate prior distribution on $f$ and other parameters of the model. The main objective of this paper is to study the contraction rate of the posterior distribution, and construct a Bayesian test for multivariate monotonicity with some desirable large sample frequentist properties. To facilitate Bayesian inference, we construct a likelihood based on the working model assumption that $\varepsilon_i\stackrel{\mathrm{i.i.d.}}{\sim}\N (0,\sigma^2)$, although the actual distribution may be non-normal. For a given $f$, let $p_{f,\sigma}(y|x)= (\sqrt{2\pi} \sigma)^{-1} \exp[-(y-f(x))^2/(2\sigma^2)]$ stand for the conditional density of $Y$ given $X=x$.  

Let $G_n=n^{-1}\sum_{i=1}^n \delta_{\bX_i}$ denote the empirical distribution of $\bX$. For a deterministic predictor variable $\bX$, this is a sequence of deterministic distributions, while for a random $X$, this sequence is random. Let $f_0$ stand for the true value of the regression function $f$, $\sigma_0$ stand for the true value of $\sigma$, and let $\P_0$ denote the true distribution of $(\bX,Y)$. The expectation with respect to  $\P_0$ will be denoted by $\E_0$.

The usual approach to Bayesian inference for model \eqref{formula:model} with $f\in \mathcal{M}$ would be to put a prior on $f$ supported within $\mathcal{M}$, and obtain the posterior distribution to make an inference. However, the shape restriction in $\mathcal{M}$ forbids certain natural priors, such as the one on step functions with the step-heights independently normally distributed, which allows fast calculations through conjugacy. A compliant prior will have to maintain the order restriction on the step-heights, which makes the posterior computation more challenging. More importantly, frequentist analyses such as posterior contraction rates and limiting coverage of credible regions are extremely hard. The projection-posterior approach provides a simple tool to ``correct'' a non-compliant posterior distribution by projecting posterior samples on the relevant parameter space and uses the resulting induced distribution to make inference, as in Lin and Dunson \cite{lin2014bayesian} and Chakraborty and Ghosal \cite{chakraborty2021convergence, ChakrabortyDensity}. A generalization of this approach using a broader ``immersion map'' was used by Wang and Ghosal \cite{Kang} to study the coverage of a Bayesian credible interval of a multivariate monotone regression function at a given point.  

To obtain posterior contraction rate in terms of a global metric like an  $\LL_1$-distance, we follow the projection-posterior approach, as in the univariate case of Chakraborty and Ghosal \cite{chakraborty2021convergence}. Given $J\in \ZZ_{>0}$, let 
$I_{\bm{1}} =  [0,J^{-1}]^d$ and $I_{\bj} = \prod_{k=1}^d ((j_{k}-1)/J, j_k/J]$ for $ \bj \in [\bm{1}:\bJ]\backslash \{\bm{1}\}$. 
Let $\cF_J=\{f:f = \sum_{\bj\in[\bm{1}:\bJ]} \theta_{\bj}\Ind_{I_{\bj}}, \theta_{\bj}\in\RR\}$, the set of piecewise constant functions. If $f$ were unrestricted, a conjugate prior for the model \eqref{formula:model} is given by letting 
\begin{align} 
	\label{prior theta}
	\theta_{\bj}\stackrel{\mathrm{ind}}{\sim} \N(\zeta_{\bj}, \sigma^2\lambda_{\bj}^2), \qquad \bj\in[\bm{1}:\bJ],
\end{align}
where $\zeta_{\bj}, \lambda_{\bj}$ are hyperparameters, and then either by choosing $J$ deterministically (increasing with $n$) or by putting a prior on $J$. The prior and the resulting posterior are both supported within $\cF_J$, and the posterior is given by 
\begin{align}
	\label{eq:posterior}
	\theta_{\bj}|(\mathbb{D}_n, \sigma^2, J) \stackrel{\mathrm{ind}}{\sim}  \N((N_{\bj} \bar{Y}|_{I_{\bj}} + \zeta_{\bj}\lambda_{\bj}^{-2})/(N_{\bj} + \lambda_{\bj}^{-2}), \sigma^2/(N_{\bj} + \lambda_{\bj}^{-2})),
\end{align}
where $N_{\bj} =\sum_{i=1}^n \mathbbm{1} \{ \bX_i \in I_{\bj}\}$, the number of observed points falling in the $\bj$th block,  and $\bar{Y}|_{I_{\bj}} = \sum_{i=1}^n Y_{i}\Ind\{\bX_{i}\in I_{\bj}\}/N_{\bj}$, $\bj\in[\bm{1}:\bJ]$. 
The resulting posterior for $f$ will be referred to as the ``unrestricted posterior'', which is not supported within $\mathcal{M}$. The projection map then produces an induced distribution suitable for an inference, to be referred to as the ``projection-posterior'' distribution.  

To study the asymptotic properties of the posterior distribution of $f$ in the setting of a deterministic predictor, we consider the $\LL_1(G_n)$-distance, while for a random predictor arising from a distribution $G$, we also use the $\LL_1(G)$-distance. It will be seen that the projection posterior inherits the convergence properties of the original posterior if the same metric is used to obtain the projection, and hence it will be sufficient to study the unrestricted posterior, which can be done using traditional tools like moment bounding or by applying the general theory of posterior contraction (cf., Ghosal and van der Vaart \cite{ghosal2017}). For random predictors, another alternative is to use the Lebesgue $\LL_1$-distance. If $G$ admits a density bounded above and below, then  the $\LL_1(G)$-distance and the Lebesgue $\LL_1$-distance are equivalent, and hence they lead to the same rate. It is also sensible to consider $\LL_p$-distances for $p$ different from $1$, but the weaker $\LL_p$-approximation property (see Lemma~\ref{approximation}) will lead to a suboptimal contraction rate $n^{-1/(pd+2)}$ for $1< p\le 2$. For the univariate case $d=1$, Chakraborty and Ghosal \cite{chakraborty2021convergence} improved the rate to the optimal rate $n^{-1/3}$ up to a logarithmic factor by using variable knots and by putting a prior on the knots, but the corresponding improved approximation result does not seem to be obtainable in the multivariate case. 

We make the following assumption throughout. 

\begin{assumption}[Design]
	\label{assumption:predictor}
	 The predictor variables  $\bX_1,\ldots,\bX_n$ are deterministic and that $\max \{ N_{\bj}: \bj\in [\bm{1}:\bJ]\}\lesssim n/J^d$, or are sampled i.i.d. from a distribution $G$ with bounded density $g$. 
\end{assumption}

\begin{assumption}[Data]
	\label{assumption:data}
	The true regression function $f_0\in \mathcal{M}$ and the true distribution of the regression error  $\varepsilon$ has mean zero and true variance $\sigma_0^2$.
\end{assumption}

\begin{assumption}[Prior]
	\label{assum:prior}
	The parameters  $\zeta_{\bj}$ and $\lambda_{\bj}$ in the prior on the coefficients $\theta_{\bj}$ satisfy $\max_{\bj}|\zeta_{\bj}|<\infty$ and $ 0<\min_{\bj}\lambda^2_{\bj}\leq \max_{\bj}\lambda_{\bj}^2 <\infty$.

	If the number $J$ of steps in each direction is not chosen deterministically, then it is given a prior supported on $\ZZ_{>0}$ satisfying the tail condition 
	\begin{align}
		\label{prior:J}
		\exp\{-b_2 J^d \log J\} \leq \pi(J) \leq \exp\{-b_1 J^d \log J\},
	\end{align}
	where $b_1$ and $b_2$ are positive hyperparameters. 
\end{assumption}

To deal with the parameter $\sigma^2$, we can plug in the marginal maximum likelihood estimator (MLE) of $\sigma^2$. Under the Gaussian working model, the marginal MLE is given by 
\begin{align}\label{sigmasq.mmle}
	\hat{\sigma}_n^2=
	\frac{1}{n}\left[\sum_{i=1}^n \big(Y_i - \sum_{\bj:\bX_i\in I_{\bj}}\zeta_{\bj}\big)^2 - \sum_{\bj\in[\bm{1}:\bJ]}\frac{N_{\bj}^2 (\bar{Y}|_{I_{\bj}} - \zeta_{\bj})^2}{N_{\bj}+\lambda_{\bj}^{-2}} \right].
\end{align}
An alternative is to adopt a fully Bayesian approach, endowing $\sigma^2$ with an Inverse-Gamma prior IG$(\beta_1, \beta_2)$, for some $\beta_1>0,\beta_2>0$. By conjugacy, the marginal posterior distribution is
\begin{align}\label{sigmasq.mpostd}
	\sigma^2|\mathbb{D}_n \sim \text{IG}(\beta_1 + n/2, \beta_2 + n\hat{\sigma}^2_{n}/2).
\end{align}

Let $\mathcal{M}_J= \cF_J\cap \mathcal{M}$.
To comply with the shape constraints, we project the posterior of $f$ onto the monotone function space $\mathcal{M}_J$ through the map 
\begin{align}
	f\mapsto f^*\in\argmin \{ \rho(f,h): h\in \mathcal{M}_J\}, 
\end{align}
provided the minimizer exists, where $\rho$ is the metric of interest. For $f= \sum_{\bj\in[\bm{1}:\bJ]} \theta_{\bj}\Ind_{I_{\bj}}\in \cF_J$, the condition of monotonicity is equivalent to that the array of the coefficients lies in the convex cone 
\begin{align}\label{convexcone}
	\mathcal{C} = \{\bm{\theta}=(\theta_{\bj}: \bj\in[\bm{1}:\bJ]): \theta_{\bj_1}\leq \theta_{\bj_2},\text{ if } \bj_1\preceq\bj_2\}. 
\end{align}
In this paper, $\rho $ will be taken as the $\LL_p(G^*)$-distance for a distribution $G^*$ on $[0,1]^d$, possibly depending on $n$ (such as $G_n$), and some $p\ge 1$, usually $1$. By minimizing the $\LL_p(G^*)$-distance over $\mathcal{M}_J$, we will get the projection posterior samples, and the corresponding induced distribution as the projection-posterior distribution to make an inference.
Let the Lebesgue measure on $[0, 1]^d$ be denoted by $\lambda$.
The following result shows that the
projection posterior given by the
$\LL_p(\lambda)$-projection onto $\mathcal{M}$ charges only $\mathcal{M}_J$.
\begin{proposition}
	\label{prop:piece}
	For any $f$ in $\cF_J$ and $p\ge 1$, 
	its $\LL_p(\lambda)$-projection onto $\mathcal{M}$, $f^*$, exists, 
	and $f^*$ is also the solution of the minimization problem $\min\{\|f-h\|_{p, \lambda}: h\in \mathcal{M}_J\}$.
\end{proposition}
However, for a general distribution $G^{\ast}$, the $\LL_p(G^{\ast})$-projection of $f\in \cF_J$ onto $\mathcal{M}_J$ is not necessarily the $\LL_p(G^{\ast})$-projection onto $\mathcal{M}$. That means, given $f\in \cF_J$, the minimizing problem $\min\{\|f - h\|_{p, G^{\ast}}: h\in \mathcal{M}\}$ can possess no solution in $\cF_J$, as the minimizing problem also depends on the weighting distribution $G^{\ast}$. This is different from the univariate case, where the same minimizing problem always has solutions in $\mathcal{M}_J$. 

We focus on the $\LL_p(G^*)$-projection onto $\mathcal{M}_J$. For $f\in \cF_J$, the minimizing problem then becomes,
\begin{align}\label{ming}
    \min_{\bm{\theta}^{\ast}\in \mathcal{C}} \sum_{\bj\in [\bm{1}:\bJ]}|\theta_{\bj} - \theta^{\ast}_{\bj}|^p G^{\ast}(I_{\bj}).
\end{align}
The solution of isotonic optimization problem in \eqref{ming} is available in some R packages like `isotone', see de Leeuw \cite{Leeuwisotone}. It is a convex optimization problem with a set of linear constraints in \eqref{convexcone}, so a general convex optimization algorithm, such as an active-set method or an interior-point method, can be applied. However, algorithms specially designed for isotonic regression may obtain the solution faster. By the algorithms given in Stout \cite{Stout2013}, problem \eqref{ming}
can be solved  in $O(J^d\log J)$ steps when $d=2$, and in $O(J^{2d}\log J)$ steps when $d\ge 3$. It is clear that the solution is unique if $p>1$ and $G^*(I_{\bj})>0$ for all $\bj$, by the strict convexity of the $\LL_p(G^*)$-norm. For the $\LL_1(G^*)$-norm, the solution may not be unique, but any solution may be chosen to define the projection-posterior. The convergence properties are not affected by the choice. 
For the choice $G^*=G_n$ primarily used in this paper, the minimization in \eqref{ming} reduces to
\begin{align}
	\label{mingn}
	\min_{\theta^{\ast}\in\mathcal{C}} \sum_{\bj\in[\bm{1}:\bJ]} N_{\bj}|\theta_{\bj}^{\ast} - \theta_{\bj}|^p,
\end{align}
while the use of the Lebesgue measure leads to the unweighted isotonization problem of the minimization of $ \sum_{\bj\in[\bm{1}:\bJ]} |\theta_{\bj}^{\ast} - \theta_{\bj}|^p$ subject to the restriction that $\bm{\theta_{\bj}}^{\ast}\in\mathcal{C}$. 

\section{Main results}\label{sec:main}

Let a sample from the projection-posterior defined by the minimization of an  $\LL_1$-distance, be denoted by $f^{\ast} = \sum_{\bj\in[\bm{1}:\bJ]} \theta_{\bj}^{\ast} \Ind_{\bj}$. The first part of the following theorem under abstract conditions gives the projection-posterior contraction rates with respect to a variety of $\LL_1$-metrics. In the second part of the theorem, the conclusion is specialized to the empirical $\LL_1$-metric or the $\LL_1$-metric with respect to the distribution of the predictor under easily verifiable conditions. 

\begin{theorem}\label{thm:l1concentration}
Let $J$ be deterministic, Assumptions {\rm  \ref{assumption:data}--\ref{assum:prior}} hold and 
    let $G^{\ast}$ be a distribution on $[0,1]^d$ possibly depending on $n$ and $\bX_1,\ldots,\bX_n$ satisfying the conditions that 
    \begin{align}
    \label{abstract G*}
\E \big[  \max_{\bj\in [\bm{1}:\bm{J}]}  G^\ast (I_{\bj}) \big] \lesssim J^{-d},\qquad    \E \big [ \sum_{\bj\in [\bm{1}:\bm{J}]} G^*(I_{\bj}) (N_{\bj}+1)^{-1} \big]\lesssim J^d/n.
    \end{align}
    Let $f^{\ast}$ be the $\LL_1(G^{\ast})$-projection of $f$ sampled from the unrestricted posterior on $\mathcal{F}_J$. Assume further that either $\sigma$ is known, or a consistent estimator is plugged-in, or its posterior distribution is consistent. Then for $\epsilon_n=\max\{\sqrt{J^d/n},J^{-1}\}$, we have that 
    \begin{align}
    \E_0 \Pi(\|f^{\ast}-f_0\|_{1,G^{\ast}} > M_n \epsilon_n | \mathbb{D}_n ) \to 0 
    \mbox{ for any } M_n\to \infty.
    \label{rate assertion}
    \end{align}
    The optimal $\LL_1(G^*)$-rate $n^{-1/(2+d)}$ is obtained above by choosing $J\asymp n^{1/(2d+1)}$.
    
    Further, let  Assumption {\rm  \ref{assumption:predictor}} hold, and if the predictor is random, assume $J^d (\log n)/n\to 0$. 
    Then the assertion \eqref{rate assertion}  holds for $G^*$ the empirical distribution $G_n$ for both deterministic and random predictor, and also for $G^*=G$ if the predictor is random with distribution $G$. 
\end{theorem}

The optimal rate above reduces to the $\LL_1$-optimal rate $n^{-1/3}$ in the univariate case obtained by Chakraborty and Ghosal \cite{chakraborty2021convergence}. 
We may also like to study the posterior contraction rate with respect to the $\LL_p$-metric. However, for $p>1$, the $\LL_p$-approximation rate by the step function $f_J$ is weaker, only $J^{-1/p}$, at monotone functions with jumps; see Remark~\ref{rem:Lp}. Hence the $\LL_p$-contraction rate of the corresponding procedure will be suboptimal.

The distribution of a random predictor $\bX$ is often unknown, but we can compute the $\LL_1(G_n)$-projection. The following corollary asserts that for random predictors with density bounded and bounded away from $0$, the 
$\LL_1(G_n)$-projection posterior achieves
the same  posterior contraction rate with respect to the $\LL_1(\lambda)$-metric (and hence also under the $\LL_1(G)$-metric, which is equivalent under the assumed condition).

\begin{corollary}
\label{corollary l1}
Let $\bX_1,\ldots,\bX_n$ be i.i.d. with distribution $G$ admitting a density function $g$ bounded and bounded away from $0$. Let $J$ be deterministic, $J\to \infty$ and $J^{d} (\log n)/n\to 0 $.
	Then under Assumptions {\rm  \ref{assumption:data}} and {\rm \ref{assum:prior}}, for $\epsilon_n = \max\{\sqrt{J^d/n}, J^{-1}\}$ and any $M_n\to \infty$, 
		$\E_0 \Pi(\|f^{\ast}-f_0\|_{1,\lambda} > M_n \epsilon_n | \mathbb{D}_n ) \to 0$  where $f^{\ast}$ is the $\LL_1(G_n)$-projection of $f$ sampled from the unrestricted posterior. 
\end{corollary}

Next, we shall construct a Bayesian test for the multivariate coordinatewise monotonicity.  A natural Bayesian test is based on the posterior probability of the region under the null hypothesis, that is, reject the hypothesis if $\Pi(f\in \mathcal{M}|\mathbb{D}_n)$ is less than $0.5$. However, such a test cannot be consistent since, non-monotone functions will also lie in any neighborhood of a monotone function, so posterior consistency does not imply that the test will be consistent. In numerical experiments, we observe that the Lebesgue $\LL_1$-distance between a sample from the unrestricted posterior and the set $\mathcal{M}$ is often positive for sample size up to $1000$. To avoid a false  
rejection of the null hypothesis, we enlarge the class of monotone functions to include functions separated by a distance at most $\delta_n$, where $\delta_n$ decreases with $n$ appropriately. Then we consider the posterior probability of the enlarged set, $\Pi(\rho(f, \mathcal{M})\leq \delta_n|\mathbb{D}_n)$, where $\rho$ is a suitable metric, usually an $\LL_1$-distance. This idea was also pursued in
Salomond \cite{salomond2014adaptive} and Chakraborty and Ghosal \cite{chakraborty2021convergence} for Bayesian tests for monotonicity in the univariate case, respectively using the $\LL_\infty$- and an $\LL_1$-distance. Below, we consider random predictors obtained from a fixed distribution $G$ independently. 
The following result shows that the resulting test is consistent at the null and at all fixed alternatives, and the power goes to one at an alternative belonging to a H\"older smooth class $\mathcal{H}(\alpha, L)$ (see Definition C.4 of Ghosal and van der Vaart \cite{ghosal2017}) even if the alternative  approaches the null, provided that happens sufficiently slowly.

\begin{theorem}
	\label{thm:testingdet}
	Let Assumptions~{\rm \ref{assumption:predictor}--\ref{assum:prior}} hold for a random predictor with distribution $G$, and let $\rho$ stand for the $\LL_1(G)$-distance. 
	Let $\gamma\in(0,1)$ and $M_n\to\infty$ be predetermined and $J\asymp n^{1/(2+d)}$. Then for the test $\phi_n = \Ind\{\Pi(\rho(f,\mathcal{M}_J)\leq M_n n^{-1/(d+2)}|\mathbb{D}_n)<\gamma\}$, we have 
	\begin{enumerate}
		\item [{\rm (i)}] $\E_0 \phi_n \to 0$ for any fixed $f_0\in\mathcal{M}$; 
		\item [{\rm (ii)}] $\E_0 (1-\phi_n) \to 0$ for any fixed integrable $f_0\notin \overline{\mathcal{M}}$, where $\overline{\mathcal{M}}$ is the $\LL_1(G)$-closure of $\mathcal{M}$; 
		\item [{\rm (iii)}] $\sup\{\E_0(1-\phi_n):f_0 \in \mathcal{H}(\alpha, L), \rho(f_0, \mathcal{M}) > \tau_n(\alpha) \} \to 0$, where
		\begin{align*}
			\tau_n(\alpha) =\begin{cases} 
				C n^{-\alpha/(2+d)}, & \text{ for some $C>0$ if $\alpha < 1$}, \\
				CM_n n^{-1/(2+d)}, & \text{ for any $C>1$ if $\alpha = 1$.}
			\end{cases}
		\end{align*}
	\end{enumerate}
\end{theorem}

The separation rate $n^{-\alpha/(2+d)}$ appearing above for consistency at smooth alternatives is weaker than the corresponding rate  $n^{-\alpha/(2\alpha+d)}$ for estimation. This is because the value of $J\asymp n^{1/(2+d)}$ is optimal for estimating monotone functions, but is suboptimal for estimating $\alpha$-smooth functions. The problem can be avoided simultaneously for all $\alpha\le 1$ by putting a prior on $J$ and using a larger enlargement in terms of the weaker Hellinger distance on the density 
\begin{align}
	\label{eq:density}
	p_{f,\sigma}(\bx,y)=(\sigma \sqrt{2\pi})^{-1} \exp [ -(y-f(\bx))^2/(2\sigma^2)]
\end{align}
of $(\bX,Y)$ (with respect to the product of $G$ and the Lebesgue measure) with size dependent on the random $J$ drawn from its posterior distribution. In this case, the posterior sampling is more involved as the posterior probabilities of each value of $J$ also need to be obtained, which involves computations of a large matrix and its determinant, and a stronger separation is needed in terms of the weaker Hellinger metric.

\begin{theorem}
	\label{thm:testingadaptive}
	Let $\sigma$ be known, Assumptions~{\rm \ref{assumption:predictor}--\ref{assum:prior}} hold for a random predictor with distribution $G$, 
	and Lebesgue density $g$ bounded away from zero. Assume $\varepsilon$ is sub-Gaussian.
	Let $\rho$ stand for the Hellinger metric on the density of $(\bX,Y)$ induced on the regression function, that is, 
	\begin{align} 
		\label{Hellinger metric}
		\rho^2(f_1,f_2)=2\big \{1-(2\pi \sigma^2)^{-1/2} \int \exp[-(f_1(\bx)-f_2(\bx))^2/(8\sigma^2)]dG(\bx)\big\}.
	\end{align}
	Let $J$ be given a prior satisfying \eqref{prior:J}. Consider the test $$
	\phi_n = \Ind\{\Pi(\rho(f,\mathcal{M}_J)\leq M_0 \sqrt{(J^d \log n) / n}|\mathbb{D}_n) <\gamma \},$$ for a predetermined $\gamma\in(0,1)$ and a sufficiently large $M_0>0$. Assume that $f_0$ is bounded. Then 
	\begin{enumerate}
		\item [{\rm (i)}] for any fixed $f_0\in\mathcal{M}$, $\E_0 \phi_n \to 0$; 
		\item [{\rm (ii)}] for any fixed $f_0$ integrable on $[0,1]^d$, and $f_0\notin \overline{\mathcal{M}}$, $\E_0(1-\phi_n) \to 0$, where $\overline{\mathcal{M}}$ is the $\LL_1(G)$-closure of $\mathcal{M}$;
		\item [{\rm (iii)}] for alternatives in the H\"older function class, we have for a sufficiently large constant $C>0$, 
		\begin{align*}
			\sup \{ \E_0(1-\phi_n): f_0\in \mathcal{H}(\alpha,L), \rho(f_0, \mathcal{M})>C(n/\log n)^{-\alpha/(1+2\alpha)} \} \to 0.
		\end{align*}
	\end{enumerate}
\end{theorem}

\begin{remark}\rm 
	In both results on testing, we can allow deterministic predictors with $\rho$  replaced by the $\LL_1(G_n)$-distance to derive properties (i) and (iii). This follows from a similar proof by obtaining posterior contraction with respect to the $\LL_1(G_n)$-metric using Theorem~8.26 of Ghosal and van der Vaart \cite{ghosal2017} for deterministic predictors. 
\end{remark}

\begin{remark}\rm 
	As the distribution $G$ of the random predictor is typically unknown, the tests used in Theorems~\ref{thm:testingdet} and \ref{thm:testingadaptive} are not generally computable. If $G$ admits a density also bounded away from $0$, then the $\LL_1(G)$-metric and the Hellinger metric given by \eqref{Hellinger metric} may be respectively replaced by the Lebesgue $\LL_1$-metric and by $\rho$ defined by 
	\begin{align}
	\label{Hellinger Lebesgue}
	\rho^2(f_1,f_2)=2\{1-(2\pi \sigma^2)^{-1/2} \int \exp[-(f_1(\bx)-f_2(\bx))^2/(8\sigma^2)]d\bx\}.
	\end{align}
	Then the conclusions of the theorems hold. For Theorem~\ref{thm:testingdet}, this follows by following the same arguments by using Part (iii) of Theorem~\ref{thm:l1concentration} instead of Part (ii). For Theorem~\ref{thm:testingadaptive}, we use the equivalence of the metrics \eqref{Hellinger metric} and \eqref{Hellinger Lebesgue} under the assumed condition and the equivalence of the projections. Moreover, the conclusion in Part (iii) of both theorems can be strengthened by replacing the H\"older class by the corresponding Sobolev class $\mathcal{W}(\alpha,L)$; see Definition C.6 of Ghosal and van der Vaart \cite{ghosal2017}. This is because the approximation rate $J^{-\alpha}$ for $\alpha$-smooth function by step function with $J$ intervals in each direction holds also for the more general Sobolev class, as the $\LL_2$-norm is stronger than the $\LL_1$-norm.  
\end{remark}


\section{Numerical results}\label{sec:simulation}

\subsection{Simulation for posterior contraction rate}

We conduct a numerical study to assess the finite sample performance of the projection posterior methods for the estimation of isotonic regression functions. We use the projection posterior sample mean as our Bayesian estimator and compare the empirical $\LL_1$-distance between our estimator and the true regression function with that of the least square estimator on data sets of different sizes.
We consider monotone regression functions:
\begin{itemize}
	\item $f_1(x_1, x_2) = x_1 + x_2$, 
	\item $f_2(x_1, x_2) = \exp\{x_1x_2\}$, 
	\item $f_3(x_1, x_2) = (x_1 + x_2)^2$, 
	\item $f_4(x_1, x_2) = \sqrt{x_1 + x_2}$, 
	\item $f_5(x_1, x_2) = (1 + \exp\{- 6(x_1+x_2-1)\})^{-1}$, 
	\item $f_6(x_1, x_2) = 0$. 
\end{itemize}
For each of sample size $n=100, 200,$ and $ 500$, and each regression function, 
we generate $20$ data sets from the true regression model $Y=f_0(\bX) + \varepsilon$ with $\bX$ uniformly distributed over $[0,1]^2$ and independent errors $\varepsilon\sim\N(0, 0.1^2)$.
Set $J = \ceil{n^{1/4} \log_{10}n}$, which is chosen slightly larger than the optimal one to get a better approximation in lower sample sizes. 
For each data set, we generate $M=1000$ unrestricted posterior sample functions. Then we compute the $\LL_1$- projection posterior, by the ``activeSet" function in R package ``isotone". With the projection posterior samples, we then compute the empirical $\LL_1$-distance of the projection posterior mean function and the data-generating regression function. For the least square estimator, we use the same piecewise constant representation of the regression functions to obtain a function estimator on the whole range of $\bX$ and to make a fair comparison with our method. The least squares isotonic estimator is obtained by using the R package ``isotonic.pen". We summarize the results 
in Table \ref{tab:l1rate}.

\begin{table}[!ht]
	\begin{center}
		\caption{The Lebesgue $\LL_1$-distance between the Bayesian projection posterior mean regression function (BP) and the true regression function and between 
			the least squares isotonic regression function (LS) and the true one with standard deviations across all data sets marked in the parentheses.}
		\label{tab:l1rate}
		\begin{tabular}{|c|cc|cc|cc|}
			\hline
			& \multicolumn{2}{c|}{$n=100$} & \multicolumn{2}{c|}{$n=200$} & \multicolumn{2}{c|}{$n=500$}\\
			& BP & LS & BP & LS  & BP & LS \\
			\hline
			\multirow{2}{*}{$f_1$} & 0.054  & 0.059 & 0.045 & 0.050 & 0.034 & 0.041\\
			&(0.003) &(0.005) &(0.003) & (0.003) &  (0.002) & (0.002)\\
			\hline
			\multirow{2}{*}{$f_2$} & 0.049& 0.051 & 0.040  & 0.043  & 0.030 & 0.034 \\
			&(0.004) &(0.006) & (0.004) & (0.004) &  (0.002) & (0.002)\\ \hline
			\multirow{2}{*}{$f_3$} & 0.085 & 0.089 & 0.072 & 0.074  & 0.055  & 0.058 \\
			& (0.006)&(0.011)& (0.004)& (0.004) &  (0.002) & (0.002)\\ \hline
			\multirow{2}{*}{$f_4$} & 0.040  & 0.045 & 0.032  & 0.038 & 0.024  & 0.030 \\
			& (0.003) & (0.004) & (0.003) & (0.004)&  (0.002) &  (0.002)\\ \hline
			\multirow{2}{*}{$f_5$} & 0.051 & 0.052  & 0.041 & 0.044  & 0.032  & 0.044\\
			& (0.005) & (0.006) &  (0.003) & (0.002) &  (0.002) &  (0.002)\\ \hline
			\multirow{2}{*}{$f_6$} & 0.032 & 0.021 & 0.026  & 0.018  & 0.021  & 0.012 \\
			&  (0.006) & (0.009)&  (0.004) &  (0.007) &  (0.004) &  (0.003)\\
			\hline
		\end{tabular}
	\end{center}
\end{table}
We can see from the table the Bayesian projection posterior estimator has a smaller $\LL_1$-error than the least squares estimator except for the last case of a constant function.

\subsection{Simulation for Bayesian monotonicity testing}

To test for $H_0: f_0 \in \mathcal{M}$, we choose $J=\ceil{n^{1/4}}$, $\gamma=0.5$ and $M_n=a(\log n)^b$, where $a$ and $b$ are two parameters to be determined. 
We run the procedure on several datasets of different sizes with both coordinatewise increasing and nonincreasing regression functions.
Then we obtain the posterior samples of $\rho(f, \mathcal{M}_J)$, denoted by $d$. We fit a linear model of $\log (d n^{1/4})$ over $\log\log n$ to find the estimates of $\log a$ and b, which leads to $a=0.237$ and $b=0.234$. In the following simulation, we will choose $M_n = 0.237(\log n)^{0.234}$.

Since a test, frequentist or Bayesian, for multivariate monotonicity does not seem to exist in the literature before, we consider the following hypothesis testing procedure as the baseline method. We confine to the normal linear regression model $Y=\beta_0 + \beta_1 X_1 + \beta_2 X_2 +\varepsilon_i$, $i=1,\ldots, n$. The hypothesis testing of multivariate monotonicity for affine  functions becomes
\begin{align*}
	H_0: \beta_1\ge 0 \text{ and }\beta_2 \ge 0,  \text{ against }  H_1: \beta_1 < 0 \text{ or }\beta_2 < 0.
\end{align*}
Given the significance level $\eta = 0.05$, we use the Bonferroni adjustment since we have only two parameters to be tested. We reject the null hypothesis when any one of the t-values of $\beta_1$ and $\beta_2$ smaller than $t_{n-3, 1-\eta/2}$. 
To study the level of these two procedures, we consider functions, $f_1, \ldots, f_6$ used in the last section. For the comparison of the power performance, we consider the following nonincreasing functions on $[0,1]^2$:
\begin{itemize}
	\item $f_7(x_1, x_2) = (x_1 + x_2 - 1)^2.$ 
	\item $f_8(x_1, x_2) = 2(x_1 +x_2 -1)^3 - (x_1 + x_2 - 1).$ 
	\item $f_9(x_1, x_2) = (x_1 +x_2 -1)^3 - 0.5(x_1 + x_2 - 1).$ 
	\item $f_{10}(x_1, x_2) = \sin((x_1 + x_2)\pi).$ 
	\item $f_{11}(x_1, x_2) = x_1 - x_2.$ 
	\item $f_{12}(x_1, x_2) = \exp\{ -10(x_1+x_2-1)^2\} +x_1 + x_2.$ 
\end{itemize}
Even though the linear model is misspecified, it can summarize the overall trend of the regression function through the sign of the coefficients, and hence is appropriate. We also consider fitting a nonparametric regression using piecewise linear functions and test for the linear hypothesis that the slope coefficients on each piece in each direction are all nonnegative. Specifically, we take $J=3$ and take the partition ${I_{\bj}}$ for $\bj=(1,1),\ldots, (3,3)$. On each $I_{\bj}$, we fit a linear model and test whether any t-value of the slope coefficients $\beta_{1, \bj}$ and $\beta_{2,\bj}$ is smaller than $t_{N_{\bj}-3, 1-\eta/18}$ by the Bonferroni adjustment.

We generate $200$ datasets for each sample size $n=100, 200$, and $500$. The predictors $\bm{X}$ and $\varepsilon$ are generated in the same way as in the last subsection. For the Bayesian procedure, we generate $200$ posterior samples for each dataset and project each posterior sample to the monotone function class $\mathcal{M}$, denoting the projection posterior sample as $f^{\ast}$. Then $\rho_n(f, \mathcal{M})$ is obtained by computing $\rho_n(f, f^{\ast})$, where $\rho_n$ is the empirical $\LL_1$-distance. The results are summarized in Tables \ref{tab:level} and \ref{tab:power}.

\begin{table}[!ht]
	\begin{center}
		\caption{Percentage of rejections to the null hypothesis of Bayesian projection posterior procure (BP), linear regression procedure (LR), and piecewise linear fitting (PL) when the true regression functions are coordinatewise increasing.}
		\label{tab:level}
		\begin{tabular}{|c|ccc|ccc|ccc|}
			\hline
			& \multicolumn{3}{c|}{$n=100$} & \multicolumn{3}{c|}{$n=200$} & \multicolumn{3}{c|}{$n=500$}\\
			& BP & LR & PL & BP & LR & PL & BP & LR & PL\\
			\hline
			$f_1$ & 0 & 0 & 0 & 0 & 0 & 0 & 0 & 0 & 0\\
			$f_2$ & 0.5 & 0 & 0.5 & 0 & 0 & 0 & 0 & 0 & 0\\
			$f_3$ & 0 & 0 & 0& 0 & 0 & 0 & 0 & 0 & 0\\
			$f_4$ & 0.5 & 0 & 0 & 0& 0 & 0 & 0 & 0 & 0\\
			$f_5$ & 0 & 0 & 0.5 & 0 & 0 & 0 & 0 & 0 & 0\\
			$f_6$ & 0 & 3 & 3 & 0 & 5 & 3.5 & 0 & 7 & 3\\
			\hline
		\end{tabular}
	\end{center}
\end{table}

\begin{table}[!ht]
	\begin{center}
		\caption{Percentage of rejections to the null hypothesis of Bayesian projection posterior procure (BP), linear regression procedure (LR), and piecewise linear fitting (PL) when the true regression functions are not coordinatewise increasing.}
		\label{tab:power}
		\begin{tabular}{|c|ccc|ccc|ccc|}
			\hline
			& \multicolumn{3}{c|}{$n=100$} & \multicolumn{3}{c|}{$n=200$} & \multicolumn{3}{c|}{$n=500$}\\
			& BP & LR & PL & BP & LR & PL & BP & LR & PL\\
			\hline
			$f_7$ & 64.5 & 8.5 & 73 & 93.5 & 10 & 99.5 & 100 & 10.5& 100\\
			$f_8$ & 100 & 84.5&83 & 100 & 98.5 & 100 & 100& 100 & 100\\
			$f_9$ & 35.5& 7.6& 28.5 & 96& 94.5 & 66.5 & 100 & 100 & 100\\
			$f_{10}$ & 100 & 100 & 100 & 100& 100 & 100& 100 & 100 & 100\\
			$f_{11}$ & 100 & 100 & 98.5 & 100 & 100 & 100& 100 & 100 & 100\\
			$f_{12}$ & 35 & 0& 55&89.5 & 0 &94.5& 100 & 0&100\\
			\hline
		\end{tabular}
	\end{center}
\end{table}

We can see from Tables \ref{tab:level} and \ref{tab:power} that all three methods can control the Type I error rate of the test to a low level, even though the linear regression model is misspecified in case $f_2$ to $f_5$. That is because the coefficients should be nonnegative when we project any coordinatewise nondecreasing function onto the linear function space. 
Noting that the null hypothesis is composite in the linear regression and the piecewise linear regression methods and the coefficients of the projected linear functions of $f_2$ to $f_5$ are all strictly greater than zero, 
it is thus reasonable that the results in table \ref{tab:level} looks conservative. However, in the case, $f_6$, where the slope coefficients are zero and on the boundary of the null hypothesis, the Bonferroni adjustment seems not that conservative, giving
an error rate very close to the nominal level even in the piecewise linear fitting where there are $18$ slope coefficients to be tested. The nonparametric Bayesian test we proposed controls the Type I error at a very low level, especially when the sample size is moderately large. We can further adjust the value of $M_n$ to make the type I error close to the nominal level $0.05$ and thus a higher power would be expected. The nonparametric Bayesian method and the piecewise linear fitting method both have high power, as they can detect all kinds of violations to coordinatewise monotonicity, global or local, as the sample size increases. However for some regression functions such as  $f_{12}$, where the is a small bump in the middle of the function graph, the linear regression totally breaks down as it focuses on the global nature. The same conclusion also applies in the case $f_7$. The proposed methods enjoy power enhancement when the signal-to-noise ratio increases. We can see this by comparing cases $f_8$ and $f_9$. From these two cases, we also notice that the proposed Bayesian method has a better capability of capturing the local violation than others.

\section{Proofs}\label{sec:proofs}

\begin{proof}[Proof of Proposition \ref{prop:piece}]
	For a given $h$, let $\bar{h} = \sum_{\bj\in[\bm{1}:\bJ]}  \lambda(I_{\bj})^{-1} \int_{I_{\bj}} h \d \lambda \cdot \Ind_{I_{\bj}}$.
	Clearly, $\bar{h}\in \mathcal{M}$ if $h\in \mathcal{M}$. 
	Since $f$ is constant on $I_{\bj}$, for every $\bx\in I_{\bj}$, 
	\begin{align}
		\bigg|\frac{\int_{I_{\bj}}h \d \lambda}{\lambda(I_{\bj}) } - f(\bx)\bigg|^p
		=\frac{\big|\int_{I_{\bj}} \left(h - f\right)\d \lambda\big|^p} { \lambda(I_{\bj})^{p} } 
		\leq  \frac{\int_{I_{\bj}}|h - f|^p \d \lambda}{ \lambda(I_{\bj})} ,
		\label{formula:contraction}
	\end{align}
	by Jensen's inequality. Taking integrals on both sides of \eqref{formula:contraction} over $I_{\bj}$, it follows that  
	$\int_{I_{\bj}} \abs{\bar{h} - f}^p \d \lambda \leq \int_{I_{\bj}} \abs{h-f}^p \d \lambda$. 
	Hence the monotone projection of $f\in \cF_J$ onto $\mathcal{M}$
	also belongs to $\cF_J$.  
	The existence of $f^{\ast}$ is ensured by the convexity and the closedness of $\mathcal{C}$ and the convexity of $\LL_p$-losses.

\end{proof}

\begin{proof}[Proof of Theorem \ref{thm:l1concentration}]
Since the posterior for $\sigma$ is consistent, it is sufficient to condition on the value of $\sigma$ lying in a small neighborhood of $\sigma_0$, unless $\sigma$ is known.  
Let $f_{0,J} = \sum_{\bj\in [\bm{1}:\bJ]} f_0(\bj/\bJ)\Ind_{I_{\bj}}$. Then $f_{0, J} \in \mathcal{M}_J$.
As $f^{\ast}$ is the $\LL_1(G^{\ast})$-projection of $f$ onto $\mathcal{M}_J$ and $f_0\in \mathcal{M}$,
	\begin{align}
		\label{eq:projection concentration}
		\|f^{\ast}-f_0\|_{1,G^{\ast}} & \leq \|f^{\ast}-f\|_{1,G^{\ast}} + \|f-f_{0,J}\|_{1,G^{\ast}} + \|f_{0,J} - f_0\|_{1,G^{\ast}} 
		\nonumber\\
		&\leq 2\|f-f_{0,J}\|_{1,G^{\ast}} +\|f_{0,J} - f_0\|_{1,G^{\ast}}.
	\end{align}
	By Lemma \ref{approximation}, $\|f_{0,J} - f_0\|_{1,G^{\ast}} \lesssim J^{-1}$ as $G^{\ast}(I_{\bj})\lesssim J^{-1}$ is assumed.
	Hence it suffices to bound $\|f-f_{0,J}\|_{1,G^{\ast}}$.
	
	Without loss of generality, we assume that $N_{\bj}>0$ for all $\bj$. 
	Let $\bar{f}_{0,J} = \sum_{\bj \in [\bm{1}:\bm{J}]} \theta_{0,\bj} \Ind_{I_{\bj}}$, where 
	$\theta_{0,\bj} = N_{\bj}^{-1} \sum_{i:\bX_{i}\in I_{\bj}}f_0(\bX_i)$. 
	Then Lemma~\ref{approximation} applied twice and the triangle inequality give $\|{f}_{0,J}-\bar f_{0,J}\|_{1,G^{\ast}}\lesssim J^{-1}$. 
	Therefore it suffices to show that 
	$\E_0\Pi(\|f-\bar{f}_{0,J}\|_{1, G^{\ast}} > M_n \sqrt{J^d/n} |\mathbb{D}_n ) \to 0$.
	
	Applying the Cauchy-Schwarz inequality first and then Markov's inequality, 
	\begin{eqnarray}
		\lefteqn{\Pi(\|f-\bar{f}_{0,J}\|_{1, G^{\ast}} > M_n \sqrt{J^d/n} |\mathbb{D}_n, \sigma )}\nonumber \\
		&& = \Pi(\sum_{\bj \in [\bm{1}:\bm{J}]} G^{\ast}(I_{\bj}) \abs{\theta_{\bj}-\theta_{0,\bj}} > M_n \sqrt{J^d/n} |\mathbb{D}_n, \sigma )\nonumber\\
		&& \leq \Pi(\sum_{\bj \in [\bm{1}:\bm{J}]} G^{\ast}(I_{\bj}) \abs{\theta_{\bj}-\theta_{0,\bj}}^2 > M_n^2 {J^d}/n |\mathbb{D}_n, \sigma )\nonumber\\
		&& \leq M_n^{-2} J^{-d} \sum_{\bj \in [\bm{1}:\bm{J}]} nG^{\ast}(I_{\bj}) \E [(\theta_{\bj}-\theta_{0,\bj})^2|\mathbb{D}_n, \sigma]. 
		\label{eq:posterior variation}
	\end{eqnarray}
	We decompose 
	\begin{align}
		\label{addone}
		\E[(\theta_{\bj}-\theta_{0,\bj})^2|\mathbb{D}_n, \sigma] = \text{Var}(\theta_{\bj}|\mathbb{D}_n)+(\E(\theta_{\bj}|\mathbb{D}_n)-\theta_{0,\bj})^2. 
	\end{align}
	We observe that
	\begin{align}
		\label{addtwo}
		\sum_{\bj \in [\bm{1}:\bm{J}]} n G^{\ast}(I_{\bj}) \text{Var}(\theta_{\bj}|\mathbb{D}_n, \sigma) \le \frac{\sigma^{2}}{\min\{1, \min_{\bj} \{\lambda_{\bj}^{-2}\}\}} \sum_{\bj \in [\bm{1}:\bm{J}]} \frac{n G^{\ast}(I_{\bj})}{N_{\bj} + 1}.
	\end{align}

	From \eqref{eq:posterior}, we know
	\begin{eqnarray}
	\label{postbias}
	\lefteqn{\sum_{\bj \in [\bm{1}:\bm{J}]} n G^{\ast}(I_{\bj}) (\E(\theta_{\bj}|\mathbb{D}_n) - \theta_{0,\bj})^2} \nonumber\\
	&& = \sum_{\bj \in [\bm{1}:\bm{J}]} n G^{\ast}(I_{\bj}) \left( \frac{N_{\bj} \bar{\varepsilon}|_{I_{\bj}} +  \lambda_{\bj}^{-2}\zeta_{\bj} - \theta_{0,\bj}\lambda_{\bj}^{-2}}{N_{\bj}+\lambda_{\bj}^{-2}} \right)^2 \nonumber \\
	&& \lesssim \sum_{\bj \in [\bm{1}:\bm{J}]} \frac{n G^{\ast}(I_{\bj}) N_{\bj}^2(\bar{\varepsilon}|_{I_{\bj}})^2 }{(N_{\bj} + 1)^2} + \sum_{\bj \in [\bm{1}:\bm{J}]} \frac{n G^{\ast}(I_{\bj}) }{(N_{\bj} + 1)^2}\nonumber\\
	&& \lesssim \sum_{\bj \in [\bm{1}:\bm{J}]} \frac{n G^{\ast}(I_{\bj})}{N_{\bj} + 1}
	\end{eqnarray}
	by noting that $\E[(\bar{\varepsilon}|_{I_{\bj}})^2|\bX, \sigma]= \sigma^2/N_{\bj}$. 
	Hence, the expectations of the expressions in \eqref{addtwo} and \eqref{postbias} are bounded by a constant multiple of $J^{d}$ in view of \eqref{abstract G*}. Combining these with \eqref{eq:posterior variation} and \eqref{addone}, it follows that 
	\begin{align}
	\label{newestimate}
	  \E \Pi(\|f-\bar{f}_{0,J}\|_{2, G^{\ast}} > M_n \sqrt{J^d/n} |\mathbb{D}_n, \sigma )  \lesssim M_n^{-2},
	\end{align}
and hence	the first part of the theorem is established.
	
	If $\max \{ N_{\bj}: \bj \in [\bm{1}:\bm{J}] \}\lesssim n/J^d$, then Lemma~\ref{sigma} ensures that the estimator and the posterior for $\sigma$ are consistent. For $G^\ast=G_n$, the condition \eqref{abstract G*} holds because $n G_n(I_{\bj})(N_{\bj}+1)^{-1} \le 1$. If $\bX_1,\ldots,\bX_n$ are i.i.d. with a bounded density $g$, then $\max \{ N_{\bj}: \bj \in [\bm{1}:\bm{J}] \}\lesssim n/J^d$ by Lemma~\ref{N_j bounds}, provided that $J^d (\log n)/n\to 0$. 	If $G^{\ast} = G_n$ for either random or deterministic predictors $\bX_i$, \eqref{addtwo} is bounded by $J^d$ up to some positive constant. If $\bX_1,\ldots,\bX_n$ are i.i.d. $G$, then owing to $N_{\bj}\sim \text{Bin}(n, G(I_{\bj}))$, we have that
	\begin{align}
		\E_0[(N_{\bj} + 1)^{-1}]=
		\begin{cases}
			\frac{1-(1-G(I_{\bj}))^{n+1}}{(n+1)G(I_{\bj})}, & \text{ if }G(I_{\bj})>0;\\
			 1, & \text{ if }G(I_{\bj})=0,
		\end{cases}
		 \label{reciprocal binomial}
	\end{align}
	so that $n G(I_{\bj}) \E [ (N_{\bj}+1)^{-1}]\le 1$, implying that  \eqref{abstract G*} holds for $G^*=G$. This completes the proof of the second part of the theorem. 
\end{proof}

\begin{proof}[Proof of Corollary \ref{corollary l1}]
    For $f^*$ the $\LL_1(G_n)$-projection, by the triangle inequality,
	\begin{align*}
	    \|f^{\ast} - f_0\|_{1,G}\le \|f^{\ast} - \bar{f}_{0,J}\|_{1,G} + \|\bar{f}_{0,J}-f_0\|_{1,G},
	\end{align*}
	where $\bar{f}_{0,J} = \sum_{\bj \in [\bm{1}:\bm{J}]} \theta_{0,\bj} \Ind_{I_{\bj}}$, with  
	$\theta_{0,\bj} = N_{\bj}^{-1} \sum_{i:\bX_{i}\in I_{\bj}}f_0(\bX_i)$. 
	From Lemma \ref{approximation},
	we know that $\|\bar{f}_{0,J} - f_0\|_{1,G} \lesssim J^{-1}$ under the assumption of bounded density.
	As $f^{\ast}$ is the $\LL_1(G_n)$-projection of $f$ onto $\mathcal{M}_J$,
	from Theorem \ref{thm:l1concentration}, 
	\begin{equation}\label{fast Gn}
		\E_0\Pi(\|f^{\ast} -\bar{f}_{0,J}\|_{1,G_n}>M_n \epsilon_n|\mathbb{D}_n)\to 0,
	\end{equation}
	since we also have $\|\bar{f}_{0,J} - f_0\|_{1,G_n} \lesssim J^{-1}$ by Lemmas~\ref{N_j bounds} and  \ref{approximation}. 
	Thus it suffices to show that
	\begin{align}\label{diff GGn}
		\E_0\Pi(\left|\|f^{\ast} -\bar{f}_{0,J}\|_{1,G}-\|f^{\ast} -\bar{f}_{0,J}\|_{1,G_n}\right|>M_n \epsilon_n|\mathbb{D}_n)\to 0.
	\end{align}
Clearly, we have 
	\begin{eqnarray}\label{diffGnG}
		\lefteqn{ \left|\|f^{\ast} -\bar{f}_{0,J}\|_{1,G}-\|f^{\ast} -\bar{f}_{0,J}\|_{1,G_n}\right|}\nonumber\\
		&& \le \sum_{\bj\in[\bm{1}:\bm{J}]} G_n(I_{\bj})|\theta^{\ast}_{\bj} - \theta_{0,\bj}|\cdot\max_{\bj}|G(I_{\bj})/G_n(I_{\bj})-1|.
	\end{eqnarray}
	Under the additional condition on the lower bound for $g$, Lemma~\ref{N_j bounds} implies that the last factor is $O_{P_0}(1)$. Thus \eqref{diffGnG} is bounded by a constant multiple of $\|f^{\ast} -\bar{f}_{0,J}\|_{1,G_n}$ on an event with $\P_0$-probability tending to $1$. 
Then this claim follows from Theorem \ref{thm:l1concentration}.
	As $g$ is bounded and bounded away from $0$,
$
\|f^{\ast} - f_0\|_{1,G}\asymp\|f^{\ast} - f_0\|_{1,\lambda}.
$, then the corollary follows.

\end{proof}

\begin{proof}[Proof of Theorem \ref{thm:testingdet}]
	(i) Since $\rho(f,\mathcal{M}_J)\le \|f-f_{0,J}\|_{1,G}$, the conclusion follows from Theorem~\ref{thm:l1concentration}. 
	
	(ii) By the definition of projection and the triangle inequality,
	\begin{align}
		\label{formula:alterinequ}
		\rho(f,\mathcal{M}_J)\geq \|f_0-f^{\ast}\|_{1,G} - \|f-f_0\|_{1,G} \geq \rho(f_0, \mathcal{M}_J) - \|f-f_0\|_{1,G}.
	\end{align}
	Thus by the triangle inequality, 
	\begin{eqnarray}   
		\lefteqn{ \Pi(\rho(f,\mathcal{M})\leq M_n n^{-1/(d+2)}|\mathbb{D}_n)} \nonumber \\
		&& \le	\Pi(\|f-f_0\|_{1,G}\geq \rho(f_0, \mathcal{M}_J) - M_n n^{-1/(d+2)}|\mathbb{D}_n).
		\label{test separation}	
	\end{eqnarray} 
	Since $\rho(f_0, \mathcal{M}_J)\ge \rho(f_0, \mathcal{M})$ and the latter is a fixed positive constant, to conclude the proof, it suffices to show that the posterior for $f$ is consistent at $f_0$ in the $\LL_{1}(G)$-metric. Let $\theta_{0,\bj} = \int_{I_{\bj}} f_0 \d G /G(I_{\bj})$ and then $f_{0,J}=\sum_{\bj} \theta_{0, \bj}\Ind_{I_{\bj}}$. By the martingale convergence theorem, $\|f_0 - f_{0,J}\|_{1,G} \to 0$. Proceeding as in the proof of Theorem~\ref{thm:l1concentration}, we conclude that 
	\begin{align}
		\label{formula:alterconsis}
		\E_0 \Pi(\|f - f_{0,J}\|_{1,G} >M_n \sqrt{J^d/n}|\mathbb{D}_n) \to 0, 
	\end{align}
	so posterior consistency holds in terms of the $\LL_1(G)$-distance.
	
	(iii) For $f_0\in\mathcal{H}(\alpha, L)$, we have $\|f_0 - f_{0,J}\|_{1,G} \lesssim J^{-\alpha}$. Together with \eqref{formula:alterconsis}, which is valid even when $f_0$ is not fixed, it follows that the $\LL_1(G)$-posterior contraction rate at $f_0$ is $\max\{\sqrt{J^d/n}, J^{-\alpha}\}\asymp n^{-\alpha/(2+d)}$ for the choice $J\asymp n^{1/(2+d)}$. For $\alpha<1$, the expression on the right hand side of  \eqref{test separation} is for large $n$ bounded by	$\Pi(\|f-f_0\|_{1,G}\geq C n^{-\alpha/(2+d)}/2 )\to_{P_0} 0$, since   
	$n^{-\alpha/(2+d)}\gg n^{-1/(2+d)}$. If $\alpha=1$, the corresponding bound for the event of interest reduces to 
	$\Pi(\|f-f_0\|_{1,G}\geq (C-1)M_n n^{-1/(d+2)}/2|\mathbb{D}_n )\to_{P_0} 0$. 
\end{proof}

\begin{proof}[Proof of Theorem \ref{thm:testingadaptive}]
	With $p_{f,\sigma}$ defined by \eqref{eq:density}, the Hellinger distance between $p_{f_1,\sigma}$ and $p_{f_2,\sigma}$ is  $\rho(f_1,f_2)$ and the Kullback-Leibler divergences are given by 
	\begin{align*}
		K(p_{f_0,\sigma}; p_{f,\sigma}) = \frac{1}{2\sigma^2}\|f-f_0\|_{2,G}^2, \qquad 
		V(p_{f_0,\sigma}; p_{f,\sigma}) = \frac{1}{\sigma^2}\|f-f_0\|_{2,G}^2.
	\end{align*}
	Thus the Kullback-Leibler ball $\{f: K(p_{f_0,\sigma}; p_{f,\sigma})\leq \epsilon^2, V(p_{f_0,\sigma}; p_{f,\sigma^2})\leq \epsilon^2\}$ contains the $\LL_2(G)$-ball $\{f: \|f-f_0\|_{2,G} \leq C\epsilon\}$ for some $C>0$, and hence to study posterior contraction at a true $f_0$, it suffices to lower bound the prior probability of the latter. Since $\|f_0 - f_{0,J}\|^2_{2,G} \leq (f_0(\bm{1}) - f_0(\bm{0}))\|f_0 - f_{0,J}\|_{1,G}$, to keep $\|f_0 - f_{0,J}\|_{2,G}$ within a targeted $\epsilon$ (which may or may not depend on $n$), $J$ should be sufficiently large to make $\|f_0 - f_{0,J}\|_{1,G}\le c\epsilon^2$ for some sufficiently small $c>0$. If a value $\bar J$, possibly depending on $n$, achieves this, then using \eqref{prior:J}, we can lower bound the required $\LL_2(G)$-prior concentration by 
	\begin{eqnarray*}
		\lefteqn{ \Pi(\bar J)\Pi(\|f-f_{0,\bar J}\|_{2,G} \leq C\epsilon|J=\bar J)} \\ 
		&&\geq \Pi(\bar J)\Pi(\cap_{j=1}^{\bar J} \{|\theta_{\bj} - \theta_{0,\bj}| \leq C_1 \epsilon^2\}) \\ 
		&& \gtrsim \exp \{-b_2\bar J^d\log \bar J - C_2 \bar J^d\log (1/\epsilon)\}
	\end{eqnarray*}
	for some constant $C_1,C_2>0$. Let $J_n$ stand for a sufficiently large multiple of $(n\epsilon^2)^{1/d}$. There are two situations to be considered. If $\epsilon>0$ is fixed at an arbitrarily small number, then $\bar J$ may be chosen as a sufficiently large constant. Then the lower bound for prior concentration in $\epsilon$-neighborhood is a fixed positive number.  Hence it follows that $\Pi (J\ge J_n)/ \Pi(\|f-f_{0}\|_{2,G} \leq C\epsilon)=o(e^{-2 n\epsilon^2})$, and hence by Theorem~8.20 of Ghosal and van der Vaart \cite{ghosal2017}, $\Pi (J> J_n|\mathbb{D}_n)\to_{P_0} 0$. If $\epsilon=\epsilon_n\to 0$ is chosen so that $n\epsilon_n^2\to \infty$ and the corresponding $\bar J=\bar J_n$ satisfies $\log \bar J_n\lesssim \log n$, and it holds that $\log (1/\epsilon_n)\lesssim \log n$ and $\bar J_n^d\log n\lesssim n\epsilon_n^2$, then for the choice  $J_n=L(n\epsilon_n^2/\log n)^{1/d}$ for some sufficiently large constant $L>0$, it again follows that $\Pi (J\ge J_n)/ \Pi(\|f-f_{0}\|_{2,G} \leq C\epsilon_n)=o(e^{-2 n\epsilon_n^2})$, and hence again  by Theorem~8.20 of Ghosal and van der Vaart \cite{ghosal2017}, $\Pi (J> J_n|\mathbb{D}_n)\to_{P_0} 0$.
	
	First we establish an auxiliary estimate essential to prove the assertions (i), (ii) and (iii). We claim that for any bounded measurable $f_0$ (not necessarily monotone or smooth) and a given $\delta>0$, if 
	$\log J_n \lesssim \log n$,
	there exists a sufficiently large constant $M_0>0$ such that  
	\begin{align}
		\label{formula:p1}
		\E_0 \Pi(\|f-f_{0,J}\|_{2,G}\geq M_0 \sqrt{J^d(\log n) /n}, J\leq J_n |\mathbb{D}_n)<\delta,
	\end{align}
	when $n$ large enough. The posterior probability in the expectation of the last display can be written as 
	\begin{align}
		\label{formula:p2}
		\sum_{J=1}^{J_n} \Pi(J|\mathbb{D}_n) \Pi\big(\sum_{\bj \in [\bm{1}: \bm{J}]} (\theta_{\bj} - \theta_{0,\bj})^2 G(I_{\bj}) \geq M_0^2 J^d (\log n) / n \big|\mathbb{D}_n \big).
	\end{align}
	By Markov's inequality and Assumption~\ref{assum:prior}, 
	\begin{multline*}
		\max_{J\le J_n}\Pi\big(\sum_{\bj\in [\bm{1}: \bm{J}_n]} (\theta_{\bj} - \theta_{0,\bj})^2 G(I_{\bj}) \geq M_0^2 J^d (\log n) / n \big|\mathbb{D}_n \big) \\
		\leq \max_{J\le J_n} \frac{n}{M_0^2 J^d \log n} \sum_{\bj\in [\bm{1}: \bm{J}_n]}  G(I_{\bj})\big[\text{Var}(\theta_{\bj}|\mathbb{D}_n) + (\E(\theta_{\bj}|\mathbb{D}_n) -\theta_{0,\bj})^2\big]  
	\end{multline*}
	which is bounded in probability by a constant multiple of 
	\begin{align} 
		\label{addfour} 
		\max_{J\le J_n} \frac{n}{M_0^2 J^d \log n} \sum_{\bj\in [\bm{1}: \bm{J}]}  G(I_{\bj})[(N_{\bj}+\lambda_{\bj}^{-2})^{-1}+(\bar{Y}|_{I_{\bj}}-\theta_{0,\bj})^2 ]
	\end{align}
    It is clear that $G(I_{\bj})\asymp J^{-d}$.
	By Lemma \ref{N_j bounds},
	\begin{align*}
	    \P_0(\bigcap_{J=1}^{J_n}\{C_1n/J^d \le \min_{\bj} N_{\bj}
	    \le \max_{\bj} N_{\bj} \le C_2 n/J^d\}) \to 1,
	\end{align*}
	provided $n/J_n \gg \log J_n$, for two constant $C_1$ and $C_2>0$.
	Then $N_{\bj}\asymp n/J^d$ uniformly for all $\bj\le \bJ$ and $J$.
	By the union bound of sub-Gaussian variables (see van der Vaart and Wellner \cite{van1996weak}, Section 2.2), we have $(\bar{\varepsilon}|_{I_{\bj}})^2 \lesssim (J^d \log n) /n$ with arbitrarily high probability, provided $\log J_n \lesssim \log n$.
	As $f_0$ is bounded, we have
	$|N_{\bj}^{-1} \sum_{i:\bX_i\in I_{\bj}}f_0(\bX_i) - \theta_{0,\bj}|^2\lesssim J^d (\log n )/n$ uniformly for all $\bj$ and $J$ with high probability.
	Thus we establish the claim in \eqref{formula:p1}. 
	
	To prove (i), we observe that the $\LL_2(G)$-approximate rate is $J^{-1/2}$, and thus $\epsilon_n\asymp \bar{J}_n^{-1/2} \asymp (n/\log n)^{-1/2(d+1)}$, so $J_n \asymp (n/\log n)^{1/(d+1)}$, and  $\Pi (J> J_n|\mathbb{D}_n)\to_{\P_0} 0$. Since  $\rho(f,\mathcal{M}_J)\lesssim \rho(f,\mathcal{M}) \le \rho(f,f_0)$, the claim follows from \eqref{formula:p1}.  
	
	To prove (ii), we choose $\epsilon>0$ arbitrarily small but fixed. By the martingale convergence theorem, $\|f_0-f_{0,J_0}\|_{1,G}<\epsilon$ for any sufficiently large $J_0$. Hence $J_n$ can be chosen a sufficiently small multiple of $(n/\log n)^{1/d}$ such that $\Pi (J> J_n|\mathbb{D}_n)\to_{\P_0} 0$. Let $\cF_n^*=\bigcup_{J=1}^{J_n} \{\sum_{\bj\in [\bm{1}, \bm{J}_n]} \theta_{\bj} \Ind_{I_{\bj}}: |\theta_{\bj}|\le n \}$. Then $\Pi (f\not\in \cF_n^*)=o(e^{-cn})$ for some constant $c>0$, and the $\LL_1(G)$-covering number of $\cF_n^*$ is bounded by $J_n^d (2n/\epsilon)^{J_n^d}$. Thus the $\epsilon$-metric entropy is bounded by $J_n^d \log n\le n \epsilon^2$. Hence the posterior distribution at $f_0$ is consistent with respect to the $\LL_1(G)$-metric, by an application of the Schwartz posterior consistency theorem (cf., Theorem~6.23 of Ghosal and van der Vaart \cite{ghosal2017}). Therefore, as $\rho(f_0, \mathcal{M}_J)$ is bounded by a positive fixed constant from below, by \eqref{formula:alterinequ}, it follows that $\Pi (\rho(f,\mathcal{M}_J)\le M_0 \sqrt{(J^d \log n)/n}|\mathbb{D}_n)\to_{\P_0} 0$.

	To prove Part (iii), we observe by Lemma~\ref{approximation} that the approximation rate at an $f_0\in \mathcal{H}(\alpha, L)$ is $J^{-\alpha}$, so that $\bar J_n\asymp \epsilon_n^{-1/\alpha}$ and $\epsilon_n \asymp (n/\log n)^{-\alpha/(2\alpha+d)}$ and $J_n \asymp (n/\log n)^{1/(2\alpha+d)}$. Using the sieve $\cF_n^*$ as defined above with this choice of $J_n$, it follows that $\Pi (f\not\in \cF_n^*)=o(e^{-Cn\epsilon_n^2})$ for a given constant $C>0$. The $\epsilon_n$-metric entropy is bounded by $J_n^d \log n\lesssim n \epsilon_n^2$. Hence it follows from Theorem~8.9 of Ghosal and van der Vaart \cite{ghosal2017} that the $\LL_1(G)$-posterior contraction rate is $(n/\log n)^{-\alpha/(2\alpha+d)}$. Thus, as $\rho(f_0, \mathcal{M}_J) \ge C(n/\log n)^{-\alpha/(2\alpha+d)}$ for a sufficiently large constant $C>0$, from \eqref{formula:alterinequ} and the probabilistic bound $(n/\log n)^{1/(2\alpha+d)}$ for $J$, the conclusion follows. 
\end{proof}

\def\thesection{\Alph{section}}
\setcounter{section}{0} 
\section{Auxiliary results}
\label{sec:appendix}

\begin{lemma}
	\label{N_j bounds}
	If $\bX_1,\ldots,\bX_n$ are a random sample from a density $g$ on $[0,1]$, $J\to\infty$, and $n/J^d \gg \log J$. If $g$ is bounded, then for some constants $C>0$, 
	$$\P_0 (\max \{ N_{\bj}: \bj\in [\bm{1}: \bm{J}] \} \le C n/J^d )\to 1.$$
	If $g$ is bounded away from zero, then for some constant $C'>0$, we have
	$$\P_0 (\min \{ N_{\bj}: \bj\in [\bm{1}: \bm{J}] \} \ge C' n/J^d )\to 1.$$
\end{lemma}

\begin{proof}
	For every $\bj$, $N_{\bj}\sim \text{Bin}(n, G(I_{\bj}))$. If $g$ is bounded from above by $a$, then $G(I_{\bj})$ is bounded by $a/J^d$. Following the same argument of the proof of Lemma A.2 of Chakraborty and Ghosal \cite{chakraborty2020coverage}, we obtain that $\P_0(N_{\bj} > Cn/J^d) \le 2\exp\{ - C' n/J^d \}$ by large deviation probability. By the condition $n/\log J \gg J^d $, we have $\P_0(\max N_{\bj} > Cn/J^d) \le 2\exp\{ - C'' n/J^d \}\to 0$. The second claim follows from a similar argument.
\end{proof}

\begin{lemma}
	\label{approximation}

	
	Let $G^*$ be a probability measure on $[0,1]^d$ such that $\max \{G^*(I_{\bj}): \bj\in [\bm{1}: \bm{J}]\}\lesssim J^{-d}$. For a given $f:[0,1]^d\to \RR$ and $J$, let $f_J:[0,1]^d\to \RR$ be defined by $f_J(\bx)=\sum_{\bj\in [\bm{1}: \bm{J}]}  \theta_{\bj} \mathbbm{1}\{ \bx\in I_{\bj}\}$, $\bx\in [0,1]^d$, where $\theta_{\bj}$ is any value between $f((\bj-\bm{1})/J)$ and $f(\bj/J)$. Then $\|f-f_J\|_{p,G^*}\lesssim J^{-1/p}$. 
	Moreover, for some appropriate choices of $\theta_{\bj}$, $\bj\in [\bm{1}: \bm{J}]$, we can ensure that $f\in \mathcal{M}$.  
\end{lemma}

\begin{proof}
	For $\theta_{\bj}$ any value between $f((\bj-\bm{1})/J)$ and $f(\bj/J)$, 
	\begin{align*}
		\|f- f_{J}\|_{1,G^*} & = \sum_{\bj} \int_{I_{\bj}} |f-\theta_{\bj}|\d G^* \\ 
		&\leq \sum_{\bj} (f(\bj/J) - f((\bj-\bm{1})/J))G^*(I_{\bj}) \\ 
		&\lesssim J^{-d}\sum_{\bj} (f(\bj/J) - f((\bj-\bm{1})/J)).
	\end{align*}
	To get the upper bound of the summation in the last inequality, we first decompose the index set $[\bm{1}:\bJ]$ in the following way. For every $\bj\in [\bm{1}:\bJ]$, Let $A_{\bj}$ be the largest possible subset of $[\bm{1}:\bJ]$ in the form $\{\ldots, \bj - 2\cdot\bm{1},\bj - \bm{1}, \bj, \bj + \bm{1}, \bj + 2\cdot\bm{1},\ldots\}$, which is a chain with respect to the coordinatewise partial order on the index set. Then we count the number of different $A_{\bj}$. Note that $A_{\bj}$ can be identified by its minimal element. The minimal element of $A_{\bj}$ should satisfy that at least one of its coordinates is $1$, otherwise we can subtract this element by $\bm{1}$ while the smaller element is still in $[\bm{1}:\bm{J}]$, thus should be in $A_{\bj}$, contradicting the fact of the minimal element. The number of different minimal elements is no larger than $d J^{d-1}$, by choosing a coordinate equal to $1$ among all $d$ coordinates and setting the rest ones free in $\{1,\ldots, J\}$. The construction of $A_{\bj}$ gives $\sum_{\bm{l} \in A_{\bj}}(f(\bm{l}/J) - f((\bm{l}-\bm{1})/J)) \leq f(\bm{1}) - f(\bm{0})$. Then we have $\|f- f_{J}\|_{1,G^*}\lesssim J^{-d}(d J^{d-1}(f(\bm{1})-f(\bm{0})))\lesssim J^{-1}$. 
	
	The monotonicity constraint will be maintained by choosing $\theta_{\bj}=\int_{I_{\bj}} f(\bx)d\bx/G(I_{\bj})$ for $\bj\in [\bm{1},\bm{J}]$, or $\theta_{\bj}=f((\bj-\bm{1})/J)$, for instance. 
	
	For $p>1$, note that $\|f- f_{J}\|_{p,G^*}^p \le (f(\bm{1}) - f(\bm{0}))^{p-1} \|f- f_{J}\|_{1,G^*}$. Then the conclusion follows.
\end{proof}

\begin{remark}\rm
    \label{rem:Lp}
    For $p>1$, the $\LL_p$ approximation rate in Lemma~\ref{approximation} may not be improved. 
    To see this, consider $f=\sum_{j=1}^d \Ind\{j: x_j>c_j\}$, where $\bm{c}$ is a fixed vector with irrational coordinates in $[0,1]$. Note that $\bm{c}$ is never on the boundary of any hypercube used for partitioning. Clearly, $f$ is a multivariate monotone function with discontinuity at any $\bx$ that shares a coordinate with $\bm{c}$. Let $\bj^*$ be the index such that $\bm{c}\in I_{\bj^*}$.  and generally for a given $J$, for $k=1,\ldots,d$, $\min \{c_{j^*_k} -(j_k^*-1)/J, j_k^*/J -c_{j^*_k}\}\gtrsim 1/J$. For any hypercube $I_{\bj}$ used in the partition such that $j_k=j^*_k$ for some $k=1,\ldots,d$, there is a jump of size at least $1$ within $I_{\bj}$. Hence, no matter how $\bm{\theta}$ is chosen, $\int_{I_{\bj}} |f-f_J\|_p^p\gtrsim J^{-d}$ for all such hypercubes. The number of hypercubes with this property is of the order $J^{d-1}$, and hence it follows that $\int  |f-f_J\|_p^p\gtrsim J^{-1}$. This shows that the approximation order cannot improve using equispaced knot points to form the hypercubes for the piecewise constant approximation. 
\end{remark}

\begin{remark}\rm 
	\label{rem:G_n satisfies bounds}
	In view of Lemma~\ref{N_j bounds}, if $J^d (\log n)/n\to 0$, then the empirical distribution satisfies the condition  $\max \{G_n(I_{\bj}): \bj\in [\bm{1}: \bm{J}]\}\lesssim J^{-d}$ in probability, and hence $\|f-f_J\|_{1,G_n} \lesssim J^{-d}$, and the implicit constant of proportionality in $\lesssim$ does not depend on $f$. 
\end{remark}

\begin{lemma} 
	\label{sigma}
	Suppose $J$ is deterministic and satisfies $J\to \infty$ and $J^d/n \to 0$. For $\bX$ either deterministic or random, 
	under Assumptions \ref{assumption:predictor}-\ref{assum:prior}, we have
	\begin{enumerate}
		\item[{\rm (i)}] $\hat{\sigma}^2_n$ converges in probability to $\sigma_0^2$ at the rate of $\max\{n^{-1/2}, J^d/n, J^{-1}\}$. 
		\item[{\rm (ii)}] If we endow $\sigma^2$ with an Inverse-Gamma prior IG($\beta_1,\beta_2$) for some $\beta_1>0,\beta_2>0$, $\sigma^2$ contracts around $\sigma^2_0$ as the same rate $\max\{n^{-1/2}, J^d/n, J^{-1}\}$.
	\end{enumerate}
\end{lemma}

\begin{proof} 
    Let $\theta_{0,\bj} = N_{\bj}^{-1}\sum_{i:X_i\in I_{\bj}} f_0(X_i)$.
	By \eqref{sigmasq.mmle}, 
	\begin{align*}
	    \hat{\sigma}^2_n = & \frac{1}{n}\sum_{i=1}^n \varepsilon^2_i + \frac{1}{n}\sum_{i=1}^n(f_0(\bX_i)-\theta_{0,\ceil{\bX_i J}})^2 + \frac{1}{n}\sum_{\bj\in[\bm{1}:\bm{J}]}N_{\bj}(\theta_{0,\bj} - \zeta_{\bj})^2\\
	    & + \frac{2}{n}\sum_{i=1}^n \varepsilon_i(f_0(\bX_i)-\theta_{0,\ceil{\bX_i J}})  + \frac{2}{n}\sum_{\bj\in[\bm{1}:\bm{J}]}N_{\bj}\bar{\varepsilon}|_{I_{\bj}}(\theta_{0,\bj} - \zeta_{\bj})\\
	    & + \frac{2}{n}\sum_{i=1}^n(f_0(\bX_i)-\theta_{0,\ceil{\bX_i J}})(\theta_{0,\ceil{\bX_i J}} - \zeta_{\ceil{\bX_i J}}) \\
	    & - \frac{1}{n}\sum_{\bj\in[\bm{1}:\bm{J}]}\frac{N_{\bj}^2(\theta_{0,\bj} - \zeta_{\bj})^2 + N_{\bj}^2(\bar{\varepsilon}|_{I_{\bj}})^2 + 2N_{\bj}^2\bar{\varepsilon}|_{I_{\bj}}(\theta_{0,\bj} - \zeta_{\bj}) }{N_{\bj} + \lambda_{\bj}^{-2}}\\
	    = & \frac{1}{n}\sum_{i=1}^n \varepsilon^2_i + \frac{1}{n}\sum_{i=1}^n(f_0(\bX_i)-\theta_{0,\ceil{\bX_i J}})^2 + \frac{2}{n}\sum_{i=1}^n \varepsilon_i(f_0(\bX_i)-\theta_{0,\ceil{\bX_i J}})\\
	    & + \frac{2}{n}\sum_{i=1}^n(f_0(\bX_i)-\theta_{0,\ceil{\bX_i J}})(\theta_{0,\ceil{\bX_i J}} - \zeta_{\ceil{\bX_i J}})\\
	    & + \frac{1}{n} \sum_{\bj \in[\bm{1}:\bJ]} \frac{\lambda_{\bj}^{-2}N_{\bj}(\theta_{0,\bj}-\zeta_{\bj})^2}{N_{\bj} + \lambda_{\bj}^{-2}}
	    + \frac{2}{n} \sum_{\bj \in[\bm{1}:\bJ]}
	    \frac{\lambda_{\bj}^{-2}N_{\bj}\bar{\varepsilon}|_{I_{\bj}}(\theta_{0,\bj}-\zeta_{\bj})}{N_{\bj} + \lambda_{\bj}^{-2}}\\
	     & + \frac{1}{n}\sum_{\bj\in[\bm{1}:\bm{J}]}\frac{ N_{\bj}^2(\bar{\varepsilon}|_{I_{\bj}})^2 }{N_{\bj} + \lambda_{\bj}^{-2}}.
	\end{align*}
	Note that $\lambda_{\bj}^{-2}$, $\zeta_{\bj}$ and $f_0$ are all bounded.
	Then we can bound
	$|\hat{\sigma}^2_n - \sigma^2_0|$  up to a constant by
	\begin{equation}\label{formula:sigmasqdecomp}
		\begin{aligned}
			&\big|\frac{1}{n}\sum_{i=1}^n \varepsilon^2_i - \sigma^2_0\big|
			+ \frac{1}{n}\sum_{i=1}^n|f_0(\bX_i)-\theta_{0,\ceil{\bX_i J}}| + \frac{1}{n}\sum_{\bj\in[\bm{1}:\bm{J}]}(\theta_{0,\bj} - \zeta_{\bj})^2\\
			&
			+\frac{1}{n}\big| \sum_{\bj\in[\bm{1}:\bm{J}]}\frac{N_{\bj}\bar{\varepsilon}|_{I_{\bj}}(\theta_{0,\bj} - \zeta_{\bj})}{N_{\bj}+\lambda_{\bj}^{-2}}\big|
			+ \frac{1}{n} \sum_{\bj\in[\bm{1}:\bm{J}]}{N_{\bj}(\bar{\varepsilon}|_{I_{\bj}})^2}.
		\end{aligned}
	\end{equation}
	
	The first term of \eqref{formula:sigmasqdecomp} is $O_{\P_0}(n^{-1/2})$. 
	By the monotonicity of $f_0$, the second term is bounded by $n^{-1}\sum_{\bj\in[\bm{1}:\bJ]}N_{\bj}(f_0(\bj/\bJ) - f_0((\bj-\bm{1})/\bJ))$. 
	By Remark \ref{rem:G_n satisfies bounds}, following the same argument of the proof of Lemma \ref{approximation}, we have the second term is $O_{\P_0}(J^{-1})$ for random $\bX$ and $O(J^{-1})$ for deterministic $\bX$ under Assumption \ref{assumption:predictor}.
	The third term is bounded by a constant multiple of $J^d/n$
	since the hyperparameters $\zeta_{\bj}$ and  $\theta_{0,\bj}$ are bounded.
	Noting that $\E[(\bar{\varepsilon}|_{I_{\bj}})^2|\bX]=\sigma_0^2/N_{\bj}$, by Markov inequality, we know that the last term is $O_{\P_0}( J^d/n)$. 
	For the fourth term, by Cauchy–Schwarz inequality, we have 
	\begin{align*}
	    \big| \sum_{\bj\in[\bm{1}:\bm{J}]}\frac{N_{\bj}\bar{\varepsilon}|_{I_{\bj}}}{N_{\bj}+\lambda_{\bj}^{-2}}(\overline{f_0(\bX)}|_{I_{\bj}} - \zeta_{\bj})\big| \lesssim J^{d/2}\sqrt{\sum_{\bj\in[\bm{1}:\bm{J}]} (\bar{\varepsilon}|_{I_{\bj}})^2}=O_{\P_0}(J^d).
	\end{align*}
	Combine all of the results and the first claim follows.

	Given the first claim, we can prove the second one by following the same proof of Proposition 4.1 (b) of Yoo and Ghosal \cite{Yoo2016}.
\end{proof}

%

\bibliographystyle{plain}
\bibliography{mybib}

\end{document}